\documentclass[12pt]{amsart}
\usepackage{amsmath}
\usepackage{amsthm}
\usepackage{bbold}
\usepackage{bbm}
\usepackage{dsfont}
\newtheorem{thm}{Theorem}
\newtheorem{prop}{Proposition}

\newtheorem{lem}{Lemma}
\newtheorem{cor}{Corollary}
\newtheorem{defi}{Definition}
\numberwithin{equation}{section}
\newcommand{\nn}{\mathbb N}
\newcommand{\rr}{\mathbb R}

\DeclareMathOperator{\oc}{\xrightarrow[]{o}}
\DeclareMathOperator{\pc}{\xrightarrow[]{p}}
\newcommand{\st}{\hbox{\rm st}}

\newcommand{\fin}{\hbox{\rm fin}}
\newcommand{\Lfin}{\mathbf{fin}}
\newcommand{\Leta}{\mathbf{n}}
\newcommand{\Lopns}{\mathbf{o-pns}}
\newcommand{\opns}{\hbox{\rm o-pns}}

\begin{document}
	
\title{AMS Journal Sample}

\author{A. Ayd{\i}n$^{1,2}$, S. G. Gorokhova$^3$, H. G\"ul$^1$}

\address{$^{1}$ Middle East Technical University, Ankara, 06800 Turkey.} 
\email{{aabdulla@metu.edu.tr} and {e032209@metu.edu.tr}}

\address{$^{2}$ Mu\c{s} Alparslan University, Mu\c{s}, 49250, Turkey.}
\email{{a.aydin@alparslan.edu.tr}}

\address{$^{3}$ Sobolev Institute of Mathematics, Novosibirsk, 630090, Russia.}
\email{{lanagor71@gmail.com} and {lana@math.nsc.ru}}

\date{14.12.2016}
	
	
\title{Nonstandard Hulls of Lattice-Normed Ordered Vector Spaces}
	
\begin{abstract}
Nonstandard hulls of a vector lattice were introduced and studied in \cite{E10,E9,E7,E5,E3}. In recent paper \cite{EG}, these notions were extended to ordered vector spaces. 
In the present paper, following the construction of associated Banach-Kantorovich space \cite{E8}, we describe and investigate nonstandard hull of a lattice-normed space,
which is a generalization of nonstandard hull of a normed space \cite{Lux}. 
\end{abstract}
\maketitle	
	
\section{Introduction}

Nonstandard analysis provides a natural approach to various branches of functional analysis (see, for example, \cite{AHFL,HM,HL,E3,E6,E4,GE,E2,GKK,LW}). 
Luxemburg's construction of the nonstandard hull of a normed space (cf. \cite{Lux,HM,AHFL}) is one of the most important and elegant illustrations of the said approach. 
Recall that, given an internal normed space $(X,\|\cdot\|)$, an element $x\in X$ is called {\em infinitesimal} if $\|x\|\approx 0$ and {\em finite} if $\|x\|\le n$ for some $n\in\nn$. Denote the set of infinitesimal 
elements  and the set of finite elements  of $X$ by $\mu(X)$ and $\fin(X)$, respectively. Since $\mu(X)$ is a vector subspace of a vector space $\fin(X)$, we may define $\hat{X}$ to be a quotient 
$\fin(X)/\mu(X)$. Note that $\hat{X}$ is a real vector space and also a Banach space (cf. \cite[p.33]{HM}) under the norm defined by
$$
  \|[x]\|=\st\|x\| = \inf_{\rr}\{a\in\rr: \|x\|\le a\} \ \ \ (x\in \fin(X)).
  \eqno(1)
$$
In the case when $X={}^*Y$ for some standard normed space $Y=(Y,\|\cdot\|)$, the normed space $\hat{Y}:=\hat{{}^*Y}=(\hat{{}^*Y},\|\cdot\|)$ is called the {\em nonstandard hull} of $Y$. 
In the present paper, we develop the notion of nonstandard hull of a normed space further by generalizing it to the case of a lattice-normed space (abbreviated by LNS).  

In 1990's, Luxemburg's construction was extended to vector lattices (see \cite{E10,E9,E8}. Note that a vector lattice $E$ can be seen as the corresponding LNS
$(E,|\cdot|,E)$. Lattice-normed vector lattices (abbreviated by LNVLs) have attracted attention in \cite{E7,E5,E3,KK,K,AEEM}. 
The general theory of lattice-normed ordered vector spaces (abbreviated by LNOVSs) is still under investigation. The present paper contributes 
to the study of this theory by using nonstandard analysis, namely by using nonstandard hulls of LNOVSs normed by Dedekind complete vector lattices. 

The scheme of nonstandard analysis used below has been introduced by Luxemburg and Stroyan \cite{LS}. In our paper, we deal only with nonstandard enlargements satisfying the general saturation principle 
(such nonstandard enlargements are called {\em polysaturated} \cite[p.47]{AHFL}). Since the basic methods of nonstandard analysis are well-developed and presented in many textbooks (see, for example, \cite{Rob,Lux,HM,HL,AHFL,KK,GKK}), we refer the reader for corresponding notions and terminology to these standard sources. We also refer to \cite{Vul,SW,LZ,Za,K,AB,E3,AT} for theory of ordered vector spaces 
(abbreviated by OVSs and \cite{E7,E5,E3,E1,EG} for nonstandard hulls of vector latices and OVSs.

The structure of the paper is as follows. In Section 2, we include elementary theory of LNOVSs in parallel with theory of LNVLs (see, for example, \cite{K,AEEM}).  
In Section 3, we introduce and investigate the nonstandard hull of an LNS normed by an Dedekind complete vector lattice. This notion is closely related with 
the construction of an associated Banach-Kantorovich space \cite{E7,E5,E3}. 
In Section 4, we investigate nonstandard hulls of LNOVS. The main result here is Theorem \ref{thm 5}, that is, the nonstandard hull of a $p$-semimonotone LNOVS 
is also $p$-semimonotone with the same constant of semimonotonicity.

\section{Preliminaries}

In the present paper, all standard OVSs assumed to be real, Archime\-dean, and equipped with the generating positive cone \cite{AT}. 
We define and study certain necessary notions such as: $p$-normality and $op$-continuity of LNOVS, $p$-Levi spaces, etc. (see also \cite{AEEM} for their lattice versions)in this section. 

The following notions in lattice-normed vector spaces (abbreviated by LNSs) are motivated by their analogies in normed spaces. 

\begin{defi}[see also \cite{AEEM}]\label{def 1}
Given an LNS $(X,p,E)$ and $A,B\subseteq X$.\\ 
$(a)$ \ $A$ is said to be {\em $p$-dense in $B$} if, for any $b\in B$ and for any $0\ne u\in p(X)$, there is $a\in A$ such that $p(a-b)\le u$.\\
$(b)$ \ $A$ is said to be {\em $p$-closed} if, for any net $a_\alpha$ in $A$ such that $p(a_\alpha-x)\rightarrow 0$ in $X$ $($abbreviated by $z_\alpha\xrightarrow{p}x$$)$, 
it holds that $x\in A$.\\
$(c)$ \ $B$ is said to be the {\em $p$-closure} of $A$ if $B$ is the intersection of all $p$-closed subsets of $X$ containing $A$.
\end{defi}

In what follows, $X=(X,p,E)$ is an LNOVS. The next property is an analogy of the well-known property of normed OVSs. 
It is a direct extension of \cite[Prop.1]{AEEM} and it has a similar proof which is omitted.

\begin{prop}\label{prop 1}
Let the positive cone $X^+$ in an LNOVS $X=(X,p,E)$ be $p$-closed. Then any monotone $p$-convergent net in $X$ is $o$-convergent to its $p$-limit.
\end{prop}

We continue with further basic notions in LNOVSs, which are motivated by their analogies for vector lattices and for LNVLs (see \cite{KK,K} and \cite{AEEM}).

\begin{defi}\label{def 2}
$(a)$ A subset $A\subseteq X$ is {\em $p$-bounded} if there exists $e\in E$ such that $p(a)\leq e$ for all $a\in A$.

$(b)$ $X$ is {\em $p$-semimonotone} if there is $M\in\rr$ such that $0\le y\le x\in X$ implies $p(y)\le Mp(x)$.

$(c)$ $X$ is {\em $p$-normal} if $x_\alpha\le y_\alpha\le z_\alpha$ in $X$, $x_\alpha\xrightarrow{p}u$, and $z_\alpha\xrightarrow{p}u$
imply $y_\alpha\xrightarrow{p}u$.

$(d)$ $X$ is a {\em $p$-Levi-space} if every $p$-bounded increasing net in $X^{+}$ is $p$-convergent.

$(e)$ $X$ is {\em $op$-continuous} if $x_\alpha\oc 0$ implies that $x_\alpha\xrightarrow{p}0$.

$(f)$ $X$ is {\em $\sigma-op$-continuous} if $x_n\oc 0$ implies that $x_n\xrightarrow{p}0$.

$(g)$ A net $(x_\alpha)_{\alpha \in A}$ in $X$ is said to be {\em $p$-Cauchy} if $(x_\alpha-x_{\alpha'})_{(\alpha,\alpha')\in A\times A}\xrightarrow{p}0$.

$(h)$ $X$ is {\em $p$-complete} if every $p$-Cauchy net in $X$ is $p$-convergent.
\end{defi}

\begin{lem}\label{lemma 1}
Let $X$ be a $p$-semimonotone LNOVS with the semimonotonicity constant $M$ and $a\le x\le b$ in $X$. Then $p(x)\le 2(M+1)(p(a)\vee p(b))$.
\end{lem}

\begin{proof} 
Since $a\le x\le b$, then $0\le x-a\le b-a$ and 
$$
  p(x)-p(a)\le p(x-a)\le Mp(b-a)\le M(p(b)+p(a)).
$$  
Hence
$$
  p(x)\le M(p(b)+p(a))+p(a)\le 
 (M+1)(p(b)+p(a))\le 2(M+1)(p(a)\vee p(b)). 
$$
\end{proof}

\begin{lem}\label{lemma 2}
Let $X$ be a $p$-semimonotone LNOVS and $\pm x_\alpha\le y_\alpha\pc 0$. Then $x_\alpha\pc 0$.
\end{lem}

\begin{proof} 
By Lemma \ref{lemma 1}, $-y_\alpha\le x_\alpha\le y_\alpha$ implies $p(x_\alpha)\le 2(M+1)p(y_\alpha)$. Since $y_\alpha\pc 0$, then $2(M+1)p(y_\alpha)\oc 0$ and hence $p(x_\alpha)\oc 0$.
Thus, $x_\alpha\pc 0$.
\end{proof}

Definition \ref{def 2}$(c)$ is motivated by the property (cf. \cite[Thm.2.23]{AT}) of normal normed OVSs.
Note that, without lost of generality, one may suppose that, in Definition \ref{def 2}$(c)$, $u=0$ and $x_\alpha\equiv 0$. 
Therefore, Lemma \ref{lemma 2} ensures that any $p$-semimonotone LNOVS is $p$-normal (in particular, any LNVL is $p$-normal).
Thus, the $p$-normality coincides with the usual normality in a normed OVS $(X,p,E)=(X,\|\cdot\|)$. 
In this case $X$ is $p$-normal iff it is $p$-semimonotone (cf. \cite[Thm.IV.2.1]{Vul1}). 

It was established in \cite[Lm.2]{AEEM} that an LNVL $(X,p,E)$ is $op$-continuous iff 
$
  X\ni w_\beta\downarrow 0\Rightarrow w_\beta\pc 0.
$

In order to extend this result, we need the following lemma.

\begin{lem}\label{lemma 3}
Let an LNOVS $X=(X,p,E)$ be $p$-semimonotone and $w_{\beta}$ be a net in  $X$. If 
$ w_\beta\downarrow 0$ implies $w_\beta\pc 0$, then $X$ is $op$-continuous.
\end{lem}

\begin{proof} 
Let $x_\alpha\oc 0$. Then there are two nets $y_\beta\downarrow 0$ and $z_\gamma\downarrow 0$ in $X$ such that, for every $\beta$ and $\gamma$, there exists $\alpha_{\beta,\gamma}$ with
$$
  -y_\beta\le x_\alpha\le z_\gamma \ \ \ (\forall \alpha\ge \alpha_{\beta,\gamma}). 
$$
By Lemma \ref{lemma 1}, 
$$
  p(x_\alpha)\le 2(M+1)(p(y_\beta)\vee p(z_\gamma)) \ \ \ (\forall \alpha\ge \alpha_{\beta,\gamma}).
  \eqno(2)  
$$
By the assumption, $p(y_\beta)\oc 0$ and $p(z_\gamma)\oc 0$. Then $p(y_\beta)\vee p(z_\gamma) \oc 0$.
It follows from $(2)$ that $p(x_\alpha)\oc 0$. Therefore, $X$ is $op$-continuous.
\end{proof}

Hence we have the following result.
\begin{thm}\label{thm 1}
A $p$-semimonotone LNOVS $X=(X,p,E)$ is $op$-continuous iff, for any net $x_\alpha\in X$, the condition $x_\alpha\downarrow 0$ implies $x_\alpha\pc 0$.	
\end{thm}

Clearly, the $op$-continuity in LNOVSs is equivalent to the order continuity in the sense of \cite[2.1.4, p.48]{K}. 
For $p$-complete LNOVS, we have more conditions for $op$-continuity (see also \cite[Thm.1]{AEEM} for the LNVL-case.).

\begin{thm}\label{thm 2}
Let an LNOVS $X=(X,p,E)$ be $p$-complete and $p$-semimonotone. The following conditions are equivalent$:$
	
$(i)$ \ \ $X$ is {\em $op$-continuous}$;$
	
$(ii)$ \ if $0\leq x_\alpha\uparrow\leq x$ holds in X, then $x_\alpha$ is a $p$-Cauchy net$;$
	
$(iii)$ $x_\alpha\downarrow 0$ in X implies $x_\alpha\pc 0$.	
\end{thm}

The proof is similar with the proof of \cite[Thm.1]{AEEM}, and we omit it.

The following two results generalizes \cite[Cor.1]{AEEM}, \cite[Cor.2]{AEEM}, and \cite[Prop.2]{AEEM} respectively. 

\begin{thm}\label{thm 3}
Let an LNOVS $(X,p,E)$ be $op$-continuous, $p$-complete, and $p$-semimonotone. Then $X$ is Dedekind complete. 
\end{thm}

\begin{proof}
Assume $0\leq x_\alpha\uparrow\leq u$ then, by Theorem \ref{thm 2}(ii), $x_\alpha$ is a $p$-Cauchy net and, since $X$ is $p$-complete, there exists $x$ such that $x_\alpha\xrightarrow{p}x$. 
It follows from Proposition \ref{prop 1}, that $x_\alpha\uparrow x$, and so $X$ is Dedekind complete.
\end{proof}

\begin{thm}\label{thm 4}
Any $p$-semimonotone $p$-Levi LNOVS $(X,p,E)$ with $p$-closed $X^+$ is $op$-continuous. 
\end{thm}

The proof is similar with the proof of \cite[Cor.2]{AEEM} and therefore it is omitted.

\begin{prop}\label{prop 2}
Any $p$-semimonotone $p$-Levi LNOVS $(X,p,E)$ with $p$-closed $X^+$ is Dedekind complete.
\end{prop}

\begin{proof}
Let $0\leq x_\alpha\uparrow\leq z\in X$. Then $p(x_\alpha)\leq Mp(z)$. Hence the net $x_\alpha$ is $p$-bounded and therefore, $x_\alpha\xrightarrow{p}x$ for some $x\in X$. 
By Proposition \ref{prop 1}, $x_\alpha\uparrow x$.
\end{proof}

\section{Nonstandard hulls of LNSs and of dominated operators acting between them}

{\em Order}- and {\em regular-nonstandard hulls} of LNSs were introduced in \cite{E7} 
as certain generalizations of Luxemburg's nonstandard hull of a normed space \cite{Lux}. Here, we employ a different approach for extending of Luxemburg's construction to LNSs.
In the rest of the paper, we suppose all LNSs to be normed by Dedekind complete vector lattices. To be certain, we fix an
Dedekind complete vector lattice $E$ for the norming lattice for all LNSs in what follows.
While considering an internal LNS $(\mathcal{X},p,\mathcal{E})$, we always assume that its norming lattice is standard, i.e. $\mathcal{E}={}^*E$. 

\subsection{Some external vector spaces associated with OVSs and LNSs}
We begin with several basic constructions from \cite{E7,E1,EG}. Let $Y$ be an OVS.
Consider the following external real vector subspaces of ${}^*Y$ \cite{EG}.
$$
  \fin({}^*Y):=\{\kappa\in {}^*Y: (\exists y\in Y) -y\le\kappa\le y\},
$$
$$
  \eta({}^*Y):=\{\kappa\in {}^*Y: \inf_Y\{y\in Y:-y\le\kappa\le y\}=0\},
$$
$$
  \opns({}^*Y):=\{\kappa\in {}^*Y: \inf_Y\{y'-y: Y\ni y\le \kappa \le y'\in Y\}=0\},
$$
and $\overline{Y}:=\fin({}^*Y)/\eta({}^*Y)$. 

Let $\mathcal{X}=(\mathcal{X},p,{}^*E)$ be an internal LNS. In accordance with \cite{E7,E5}, consider the following external subspaces of $\mathcal{X}$: 
$$
  \Lfin(\mathcal{X})=\{x\in\mathcal{X}: p(x)\in\fin({}^*E)\},
$$
$$
  \Leta(\mathcal{X})=\{x\in\mathcal{X}: p(x)\in\eta({}^*E)\},
$$
In the case of a standard LNS $X=(X,p,E)$,
$$
  \Lopns({}^*X)=\{\kappa\in{}^*X: \inf_E\{p(\kappa-x): x\in X\}=0\}.
$$
Remark that, similarly to the case in which $X=(X,p,E)$ is a normed space (cf. \cite[Prop.2.2.2.]{AHFL}), it can be easily shown that $X$ is $p$-complete iff $\Lopns({}^*X)=X+\Leta({}^*X)$.

\subsection{Nonstandard hull of an LNS}
For an internal LNS $\mathcal{X}=(\mathcal{X},p,{}^*E)$, consider the quotient $\overline{\mathcal{X}}:=\Lfin(\mathcal{X})/\Leta(\mathcal{X})$ 
and define the mapping $\overline{p}:\overline{\mathcal{X}}\to E$ by the following rule motivated by the formula (1) (see also \cite[Thm.2.3.5.]{E5} and \cite[3.1]{E7}):
$$
  \overline{p}([x]):=\inf\limits_{E}\{e\in E: e\ge p(x)\} \ \ \ (x\in\Lfin(\mathcal{X})).
  \eqno(1^*)
$$
It is easy to see that this mapping is a well defined $E$-valued norm on $\overline{\mathcal{X}}$. 

\begin{defi}\label{def 3}
Given an LNS $(X,p,E)$. The LNS $(\overline{{}^*X},\overline{p},E)$ is called {\em nonstandard hull} of $(X,p,E)$.
\end{defi}

According to \cite[Thm.3.5]{E7}, nonstandard hull of $(X,p,E)$ is a Banach-Kantorovich space, when $p$ is decomposable. 
Main reason for using of the term "nonstandard hull" lies in \cite[Thm.2.4.1.]{E5} (see also \cite[Thm.4.3]{E7}) saying that, in the case of decomposable LNS $(X,p,E)$, 
$p$-completion of $(X,p,E)$ can be obtained by natural embedding of $(X,p,E)$ into $(\overline{{}^*X},\overline{p},E)$, and then by just taking its $p$-closure there.

\subsection{Nonstandard hulls of dominated operators between decomposable LNSs}
Given two LNSs $(X,p,E)$ and $(Y,m,E)$. Let $T:X\to Y$ be a dominated operator (cf. \cite[4.4.1.]{K}). Under the assumption of decomposable $X$, $T$ has an exact dominant $\mathbf{I}T\mathbf{I}$
(cf. \cite[4.4.2.]{K}) and, in that case, the space $M(X,Y)$ can be considered as a decomposable LNS $(M(X,Y),\mathbf{I}\cdot\mathbf{I},L_b(E))$.    

Denote by $\mathcal{M}_n({}^*X,{}^*Y)$ the set of all internal linear operators from ${}^*X$ into ${}^*Y$ which admit standard order-continuous dominants,
that is: $T\in M_n({}^*X,{}^*Y)$ iff there is an operator $S\in L(E,F)$ satisfying ${}^*m(T\kappa)<{}^*S({}^*p(\kappa))$ for all $\kappa\in{}^*X$.
The following lemma was proved in \cite[Lm.2.4.2.]{E5}.

\begin{lem}\label{lemma 4}
For every operator $T\in\mathcal{M}_n({}^*X,{}^*Y)$, $T(\Lfin({}^*X))\subseteq\Lfin({}^*Y)$ and $T(\Leta({}^*X))\subseteq\Leta({}^*Y)$.
\end{lem}

Lemma \ref{lemma 4} ensures that for any operator $T\in \mathcal{M}_n({}^*X,{}^*Y)$, the rule 
$$
  \overline{T}([\kappa]):=[T\kappa] \ \ \ (\kappa\in\Lfin({}^*X)) 
$$
defines a mapping $\overline{T}:\overline{X}\to\overline{Y}$. By \cite[Thm.2.4.3.]{E5},
$\overline{T}$ is a linear dominated operator from $(\overline{{}^*X},\overline{p},E)$ into $(\overline{{}^*Y},\overline{m},E)$
with the least dominant $\mathbf{I}\overline{T}\mathbf{I}$ satisfying 
$$
  \mathbf{I}\overline{T}\mathbf{I}\le\inf\{S\in L_n(E):{}^*S\ge\mathbf{I}T\mathbf{I}\},
  \eqno(3)
$$
where $\mathbf{I}T\mathbf{I}$ is the least internal dominant of $T$. The operator $\overline{T}$ is said to be {\em nonstandard hull} of $T$. 
Since $T\in M_n(X,Y)$ iff ${}^*T\in \mathcal{M}_n({}^*X,{}^*Y)$, the inequality $(3)$ implies that $\mathbf{I}\overline{{}^*T}\mathbf{I}=\mathbf{I}T\mathbf{I}$
for any $T\in M_n(X,Y)$ (see also \cite[Thm.2.4.4.]{E5}).

\section{Nonstandard hulls of LNOVSs}

\subsection{Nonstandard hull of an LNOVS}
Let $\mathcal{Y}=(\mathcal{Y},p,{}^*E)$ be an internal $p$-semimonotone LNOVS with a finite constant $\mathcal{M}\in\fin(\rr)$ of the semimonotonicity. 
The key step is the following technical lemma.

\begin{lem}\label{lemma 5}
$\Leta(\mathcal{Y})$ is an order ideal in $\Lfin(\mathcal{Y})$.
\end{lem}

\begin{proof} 
Since $\Leta(\mathcal{Y})$ is a real vector subspace of $\Lfin(\mathcal{Y})$, it is enough to show that $\Leta(\mathcal{Y})$ is order convex.
Let $\xi\le\kappa\le\zeta$ with $\kappa\in\mathcal{Y}$, and $\xi,\zeta\in\Leta(\mathcal{Y})$. By Lemma \ref{lemma 1},
$$
  p(\kappa)\le 2(\mathcal{M}+1)(p(\xi)\vee p(\zeta))\le 2(\st(\mathcal{M})+2)(p(\xi)\vee p(\zeta))\in\eta({}^*E), 
$$
and hence $\kappa\in\Leta(\mathcal{Y})$. 
\end{proof}

\begin{thm}\label{thm 5}
Let $(\mathcal{Y},p,{}^*E)$ be a $p$-semimonotone LNOVS with a finite constant $\mathcal{M}$ of semimonotonicity.
Then $\Lfin(\mathcal{Y})/\Leta(\mathcal{Y})$ is an OVS. Moreover, the LNOVS $(\overline{\mathcal{Y}},\overline{p},E)$ 
is $p$-semimonotone with a constant $M=\st({\mathcal{M}})$ of semimonotonicity.
\end{thm}

\begin{proof} 
$\overline{\mathcal{Y}}=\Lfin(\mathcal{Y})/\Leta(\mathcal{Y})$ is an OVS, by Lemma \ref{lemma 4}. Now, let $0\le[\kappa]\le[\xi]\in\overline{\mathcal{Y}}$. By the definition of 
ordering in the quotient space $\Lfin(\mathcal{Y})/\Leta(\mathcal{Y})$ (cf. \cite[p.3]{E0}), we may assume that $0\le\kappa\le\xi$. Hence $\frac{1}{\mathcal{M}}p(\kappa)\le p(\xi)$ and then,
for any $n\in\nn$,
$$
  \overline{p}([\xi])=\inf\limits_{E}\{e\in E: e\ge p(\xi)\}\ge\inf\limits_{E}\{e\in E: e\ge {\mathcal{M}}^{-1}p(\kappa)\}\ge
$$
$$
  \inf\limits_{E}\bigl{\{}e\in E: e\ge\bigl{(}\frac{1}{\mathcal{M}}-\frac{1}{2n}\bigr{)}p(\kappa)\}\ge\inf\limits_{E}\bigl{\{}e\in E: e\ge\bigl{(}\frac{1}{M}-\frac{1}{n}\bigr{)}p(\kappa)\bigr{\}}=
$$  
$$
  \bigl{(}\frac{1}{M}-\frac{1}{n}\bigr{)}\inf\limits_{E}\{e\in E: e\ge p(\kappa)\}=\bigl{(}\frac{1}{M}-\frac{1}{n}\bigr{)}\overline{p}([\kappa]).
$$  
Since the inequality is true for all $n\in\nn$, we obtain $\overline{p}([\xi])\ge (\frac{1}{M}\overline{p}([\kappa]))$ or $\overline{p}([\kappa])\le M\overline{p}([\xi])$, as desired.
\end{proof}

\begin{cor}\label{cor 1}
Let $(Y,p,E)$ be a $p$-semimonotone LNOVS. Then $(\overline{{}^*Y},\overline{p},E)$ is also a $p$-semimonotone LNOVS with 
the same constant of semimonotonicity.
\end{cor}

\begin{proof} 
Let $M$ be a semimonotonicity constant of $(Y,p,E)$. By the transfer principle, ${\mathcal{M}}=M$ is a semimonotonicity constant of $(\overline{{}^*Y},\overline{p},E)$.
Now, apply Theorem \ref{thm 5}.
\end{proof}

\begin{cor}\label{cor 2}
Let $(Y,\|\cdot\|)$ be a normal OVS. Then its nonstandard hull $\overline{{}^*Y}$ is a normal Banach space.
\end{cor}

\begin{proof}
Notice that any OVS is normal iff it is semimonotone (cf. \cite[Thm.IV.2.1.]{Vul1}) and apply Theorem \ref{thm 5}.  
\end{proof}

\subsection{Internal LNVLs}
Here we consider some properties of LNS $(\overline{\mathcal{Y}},\overline{p},E)$ in the case where $(\mathcal{Y},p,{}^*E)$ is an internal LNVL. 

\begin{thm}\label{thm 6}
Let $(\mathcal{Y},p,{}^*E)$ be an internal LNVL. Then $(\overline{\mathcal{Y}},\overline{p},E)$ is also an LNVL. 
\end{thm}

\begin{proof} 
Note that the quotient $\Lfin(\mathcal{Y})/\Leta(\mathcal{Y})$ of a vector lattice $\Lfin(\mathcal{Y})$ by an order ideal $\Leta(\mathcal{Y})$
is a vector lattice. Since $(\mathcal{Y},p,{}^*E)$ has a semimonotonicity constant $\mathcal{M}=1$ then, by Theorem \ref{thm 5}, the LNOVS 
$(\overline{\mathcal{Y}},\overline{p},E)$ has a semimonotonicity constant $M=1$, which means that $\overline{p}([\kappa])\le\overline{p}([\xi])$
for all $\kappa,\xi$ with $|\kappa|\le|\xi|$. Therefore $\overline{p}$ is a $E$-valued lattice norm on $\overline{\mathcal{Y}}$ and $(\overline{\mathcal{Y}},\overline{p},E)$ is an LNVL.    
\end{proof}

\noindent
The following proposition generalizes \cite[Prop.4.7]{HM} for LNVL.

\begin{prop}\label{prop 3}
Let $(\mathcal{Y},p,{}^*E)$ be an internal LNVL. Then the LNVL $(\overline{\mathcal{Y}},\overline{p},E)$ is $\sigma-op$-continuous, and
every monotone $p$-bounded sequence in $\overline{\mathcal{Y}}$ is order-bounded.
\end{prop}

\begin{proof} 
First we show $\sigma-op$-continuity. Clearly, it is enough to show that $\overline{\mathcal{Y}}\ni x_n\downarrow 0$ implies $\overline{p}(x_n)\downarrow 0$.  
Suppose in contrary that $\overline{p}(x_n)\downarrow\ge u$ for all $n$ and some $0\ne u\in E_+$. The monotonicity of the lattice norm $p$ ensures existence of
a sequence $\kappa_n\downarrow$ in $\mathcal{Y}_+$ with $[\kappa_n]=x_n$ and $2p(\kappa_n)\ge u$ for all $n\in\nn$. Consider the sequence of nonempty internal sets
$$
  A_n=\{\chi\in\mathcal{Y}: 2p(\chi)\ge u \ \& \ 0\le\chi\le\kappa_n\} \ \ \ (n\in\nn).
$$
By saturation principle there exists $\chi\in\bigcap\limits_{n=1}^{\infty}A_n$. Then $0<[\chi]\le[\kappa_n]=x_n$ violating $x_n\downarrow 0$.

For the second part, let $\overline{\mathcal{Y}}\ni x_n\downarrow$ and $\overline{p}(x_n)\le u\in E$ for all $n\in\nn$. The monotonicity $p$ gives a sequence $\kappa_n\downarrow$ 
in $\mathcal{Y}$ with $[\kappa_n]=x_n$ and $p(\kappa_n)\le 2u$ for all $n\in\nn$. By the saturation principle there is $\chi\in\mathcal{Y}$ with $p(\chi)\le 2u$ and 
$\chi\le\kappa_n\le\kappa_1$ for all $n\in\nn$. Hence $x_n=[\kappa_n]\in\big[[\chi],[\kappa_1]\big]$ for all $n\in\nn$ what is required. 
\end{proof}

\subsection{Nonstandard criterion for $op$-continuity}
The following theorem generalizes \cite[Thm.4.5.3.]{E5}.

\begin{thm}\label{thm 7}
An LNVL $(X,p,E)$ is $op$-continuous iff $\eta({}^*X)\subseteq\Leta({}^*X)$.
\end{thm}

\begin{proof} 
Suppose that $(X,p,E)$ is $op$-continuous and fix $\kappa\in\eta({}^*X)$. 
Then, there exists a net $x_{\alpha}\in{}^*X$ such that $x_{\alpha}\downarrow 0$ and $0\le\kappa\le x_{\alpha}$. 
Clearly, $x_{\alpha}\oc 0$, so we have $x_{\alpha}\pc 0$ since $p$ is $op$-continuous. Since $0\le p(\kappa)\le p(x_{\alpha})$, 
it follows that $p(\kappa)\in\eta({}^*E)$ or $\kappa\in\Leta({}^*X)$. Hence $\eta({}^*X)\subseteq\Leta({}^*X)$.
  
Now suppose $\eta({}^*X)\subseteq\Leta({}^*X)$ and $X\ni x_{\alpha}\oc 0$. 
Then there are two nets $y_\beta\downarrow 0$ and $z_\gamma\downarrow 0$ in $X$ such that, for every $\beta$ and $\gamma$, there exists 
$\alpha_{\beta,\gamma}$ with
$$
  -y_\beta\le x_\alpha\le z_\gamma \ \ \ \ (\alpha\ge \alpha_{\beta,\gamma}). 
$$
Thus, $x_{\alpha}\in\eta({}^*X)$ for all infinitely large $\alpha$. So, by the hypothesis, $x_{\alpha}\in\Leta({}^*X)$ for all infinitely large $\alpha$.
Therefore, $ p(x_{\alpha})\rightarrow 0$ and $(X,p,E)$ is $op$-continuous.
\end{proof}

We finish with a discussion of $p$-Levi property. Let $(X,p,E)$ be $p$-Levi and $(x_{\alpha})_{\alpha\in A}$ be a monotone $p$-bounded net in $X$.
Then $x_{\alpha}\pc x$ for some $x\in X$. By the transfer principle, $x_{\alpha}\in\Lfin({}^*X)$ for all $\alpha\in{}^*A$. Given an infinitely large $\nu$. 
Since $x_{\alpha}\pc x$, then $x_{\nu}\in x+\Leta({}^*X)\subseteq\Lopns({}^*X)$. We do not know under which conditions on $(X,p,E)$ the converse is also true.

\end{document}